\documentclass[11pt]{article}
\usepackage{geometry}                
\usepackage{graphicx}
\usepackage{amssymb}
\usepackage{epstopdf}
\usepackage{amsthm}
\usepackage{amsmath}
\usepackage{mathtools}
\usepackage{txfonts}
\newtheorem{theorem}{Theorem}[section]
\newtheorem{lemma}[theorem]{Lemma}
\newtheorem{proposition}[theorem]{Proposition}
\newtheorem{corollary}[theorem]{Corollary}
\newtheorem{conjecture}[theorem]{Conjecture}

\newcommand*{\pklambda}{\ensuremath{p_{k+1} - p_k < \lambda \sqrt{p_k}} }
\newcommand*{\pnsqrtd}{\ensuremath{\sqrt{p_{n+1}} - \sqrt{p_n} < \frac{\lambda}{(\sqrt{1.2} + 1)}} }
\newcommand*{\pnlambda}{\ensuremath{p_{n+1} - p_n < \lambda \sqrt{p_n}} }
\newcommand*{\ngn}{\ensuremath{n + g(n)\sqrt{n}} }
\newcommand*{\gpk}{\ensuremath{p_{k+1} - p_k < g(p_k)\sqrt{p_k}} }
\newcommand*{\gpkk}{\ensuremath{p_k + g(p_k)\sqrt{p_k}} }
\newcommand*{\nogno}{\ensuremath{n_0 + g(n_0)\sqrt{n_0}} }
\newcommand*{\diffpn}{\ensuremath{p_{n+1} - p_n}}
\newcommand*{\diffsqrtpn}{\ensuremath{\sqrt{p_{n+1}} - \sqrt{p_n}}}

\title{On Proving of Diophantine Inequalities with Prime Numbers by Evaluations of the Difference between Consecutive Primes}
\author{Felix Sidokhine}

\begin{document}
\maketitle

\begin{abstract}
\noindent Using as the working hypothesis of an evaluation of the difference between primes $p_{n+1} - p_n = O(\sqrt{p_n})$ we represent in detail the proofs of Legendre's and Oppermann's conjectures.
\end{abstract}

\section{Introduction}

Applying the best available evaluation of the difference between primes, $p_{n+1} - p_n = O(p_n^{0.525})$ \cite{Baker:2001aa} we have obtained proofs for some Diopantine inequalities with primes including Ingham's results \cite{Sidokhine:2014aa}. Some authors believe that in the presence of a stronger evaluations of the difference between consecutive primes it may be possible to prove Legendre's conjecture and some other statements \cite{Ribenboim:2004aa}. The generally expected evaluation of the difference between consecutive primes is $p_{n+1} - p_n = O(\sqrt{p_n})$ \cite{Heath-Brown:2005aa}, \cite{Pintz:2009aa}. In this paper, using $p_{n+1} - p_n = O(\sqrt{p_n})$ we are able to prove Legendre's and Oppermann's conjectures.

\section{$p_{n+1} - p_n = O(\sqrt{p_n})$ and Diophantine inequalities with primes} 

\begin{proposition}\label{prop_a_1}
The interval $(n,n + \lambda \sqrt{n})$ contains a prime for every integer $n \geq c(\lambda)$ where $\lambda,c(\lambda)$ are some constants, if and only if \pklambda is true for all primes $p_k \geq c(\lambda)$.
\end{proposition}

\begin{proof}
Let $(n, n + \lambda \sqrt{n})$ contain primes for every integer $n \geq c(\lambda)$. Then for $n = p_k$ the interval $(p_k,p_k + \lambda \sqrt{p_k})$ contains a prime $q$. Hence we have $p_k < q < p_k + \lambda \sqrt{p_k}$. Since $p_k < p_{k+1} \leq q$, \pklambda is true. 
\newline
\newline
Let \pklambda be true for every prime $p_k \geq c(\lambda)$. Let $n_0$ be such that $(n_0, n_0 + \lambda \sqrt{n_0})$ contains no primes. Let $p_{n-1},p_n$ be such that $p_{n-1} < n_0 < p_n$. Then $(p_{n-1},p_{n-1} + \lambda \sqrt{p_{n-1}})$ contains no primes. Since $n_0$ is not prime, the interval $(p_{n-1},n_0 + \lambda \sqrt{n_0}) = (p_{n-1},n_0) \cup [n_0] \cup (n_0,n_0 + \lambda \sqrt{n_0})$ contains no primes. Furthermore, $(p_{n-1},p_{n-1} + \lambda \sqrt{p_{n-1}}) \subset (p_{n-1},n_0 + \lambda \sqrt{n_0})$ since $p_{n-1} + \lambda \sqrt{p_{n-1}} < n_0 + \lambda \sqrt{n_0}$, so $(p_{n-1},p_{n-1} + \lambda \sqrt{p_{n-1}})$ contains no primes, contradicting \pklambda.
\end{proof}

\begin{corollary}\label{cor_a_1}
Let \pnsqrtd, where $\lambda,c(\lambda) > 25$ are constants, be true for all primes $p_n \geq c(\lambda)$. Then \pnlambda is true for every $p_n \geq c(\lambda)$. 
\end{corollary}

\begin{proof}
Since according to \cite{Nagura:1952aa} for any pair of neighbouring primes, $p_{n+1} < \frac{6}{5}p_n$, where $p_n > 25$ is true; \pnlambda  is also true for every prime $p_n \geq c(\lambda)$.
\end{proof}

\begin{corollary}\label{cor_a_2}
Let \pnsqrtd, where $\lambda,c(\lambda) > 25$ are constants, be true for all primes $p_n \geq c(\lambda)$. Then the interval $(n,n+\lambda\sqrt{n})$ contains a prime for every integer $n \geq c(\lambda)$. 
\end{corollary}

\begin{proof}
Corollary \ref{cor_a_2} is a consequence of proposition \ref{prop_a_1} and corollary \ref{cor_a_1}. 
\end{proof}

\begin{proposition}\label{prop_a_2}
The interval $(n,\ngn)$, where $g(n) = o(1)$ and $g(n)\sqrt{n}$ is a non-decreasing function, contains a prime for every integer $n \geq c(g)$, where $c(g)$ is some constant; if and only if \gpk is true for any prime $p_k \geq c(g)$. 
\end{proposition}

\begin{proof}
Let \ngn contain a prime for every integer $n \geq c(g)$. Then $(p_k, \gpkk)$, where $n = p_k$, contains a prime $q$ such that $p_k < q < \gpkk$. Since $p_k < p_{k+1} \leq q$, \gpk is true.
\newline
\newline
Let \gpk be true for every $p_k \geq c(g)$. Let $n_0$ be such an integer that the interval $(n_0, \nogno)$ contains no primes. Let $p_{n-1},p_n$ be such that $p_{n-1} < n_0 < p_n$, hence the interval $(p_{n-1}, p_{n-1} + g(p_{n-1})\sqrt{p_{n-1}})$ contains no primes. Since $n_0$ is not prime, the interval $(p_{n-1},\nogno) = (p_{n-1},n_0)\cup[n_0]\cup(n_0,\nogno)$ contains no primes. $(p_{n-1}, p_{n-1} + g(p_{n-1})\sqrt{p_{n-1}})  \subset (n_0,\nogno)$ since $p_{n-1} + g(p_{n-1})\sqrt{p_{n-1}} < \nogno$; then the interval $(p_{n-1}, p_{n-1} + g(p_{n-1})\sqrt{p_{n-1}})$ contains no primes, contradicting $\gpk$. 
\end{proof}

\begin{corollary}\label{cor_a_3}
Let $g(n) = o(1)$ and there exists a constant $c(g)$ such that the interval $(n,\ngn)$ contains a prime for every integer $n \geq c(g)$, then $\sqrt{p_{m+1}} - \sqrt{p_m} = o(1)$ is true.
\end{corollary}

\begin{proposition}\label{prop_a_3}
$\diffpn = O(f(p_n))$ is true if and only if $\diffsqrtpn = O(\frac{f(p_n)}{\sqrt{p_n}})$ is true.
\end{proposition}

\begin{proof}
Let $\diffpn = O(f(p_n))$ be true. Then there exist such constants $k, N_k$ that $\diffpn < kf(p_n)$ is true for every $p_n \geq N_k$. Hence, $\diffsqrtpn = O(\frac{f(p_n)}{\sqrt{p_n}})$ is true.
\newline
\newline
Let $\diffsqrtpn = O(\frac{f(p_n)}{\sqrt{p_n}})$ be true. Then there exist such constants $k, N_k > 25$ that $\diffsqrtpn < \frac{kf(p_n)}{\sqrt{p_n}}$ is true for any $p_n \geq N_k$. Then, $\diffpn < (\sqrt{1.2} + 1)kf(p_n)$ is true according to \cite{Nagura:1952aa}, and therefore $\diffpn = O(f(p_n))$ is also true.
\end{proof}

\begin{proposition}\label{prop_a_4}
Let Cramer's conjecture be true, then there exists some infinite subset of primes $E$ such that for every prime $p_n \in E$, $\frac{\ln(p_n)}{\sqrt{p_n}} < \diffsqrtpn < \frac{k \log^2(p_n)}{\sqrt{p_n}}$ is true.
\end{proposition}

\begin{proof}
According to Cramer's conjecture \cite{Cramer:1936aa}, $\diffpn = O(\log^2(p_n))$ and proposition \ref{prop_a_3}, $\diffsqrtpn = O(\frac{\log^2(p_n)}{\sqrt{p_n}})$ is true. Then, there exist such $k, N_k$ that for every $p_n \geq N_k$, $\diffsqrtpn < \frac{k\log^2(p_n)}{\sqrt{p_n}}$ is true. Furthermore, $\diffpn = O(\log(p_n))$ is not true according to E. Westzynthius and so $\diffsqrtpn = O(\frac{\ln(p_n)}{\sqrt{p_n}})$ is also not true according to proposition \ref{prop_a_3}. Therefore there exists an infinite set of primes $S$ such that $\frac{\ln(p_n)}{\sqrt{p_n}} < \diffsqrtpn$ is true. Taking $E = \{ p_n \in S | p_n \geq N_k \}$, the inequality $\frac{\ln(p_n)}{\sqrt{p_n}} < \diffsqrtpn < \frac{k \log^2(p_n)}{\sqrt{p_n}}$ is true for any $p_n \in E$.
\end{proof}

\section{Legendre's conjecture}

\begin{conjecture}[Legendre]\label{legendre_conjecture}
The interval $(n^2,(n+1)^2)$ contains a prime for any $n \in \mathbb{N}$.
\end{conjecture}


\begin{lemma}[]\label{lemma_b_1}
The interval $(n - 2\sqrt{n},n)$ contains a prime for all $n \geq 4$ if and only if $p_k - p_{k-1} < 2\sqrt{p_k}$ is true for all $p_k \geq 3$.
\end{lemma}


\begin{proof}
Let $(n - 2\sqrt{n},n)$ contain a prime, then the interval $(p_k - 2\sqrt{p_k},p_k)$ where $n = p_k$ contains a prime $q$. Therefore $p_k -2\sqrt{p_k} < q < p_k$, and since $q \leq p_{k-1} < p_k$, $p_k - p_{k-1} < 2\sqrt{p_k}$ is true.
\newline
\newline
Let $p_k - p_{k-1} < 2\sqrt{p_k}$ be true for all $p_k \geq 3$, but there exists such $n_0$ that $(n_0 - 2\sqrt{n_0},n_0)$ contains no primes. Let $p_{n-1},p_n$ be such primes that $p_{n-1} < n_0 < p_n$. Then the interval $(p_n - 2\sqrt{p_n},p_n)$ contains no primes. Since $n_0$ is not prime, the interval $(n_0 - 2\sqrt{n_0}, p_n) = (n_0 - 2\sqrt{n_0},n_0) \cup [n_0] \cup (n_0,p_n)$ contains no primes. Moreover, $(p_n - 2\sqrt{p_n},p_n) \subset (n_0 - 2\sqrt{n_0},p_n)$ since $n_0 - 2\sqrt{n_0} < p_n - 2\sqrt{p_n}$ so the interval $(p_n - 2\sqrt{p_n},p_n)$ contains no primes, contradicting $p_k - p_{k-1} < 2\sqrt{p_k}$. 
\end{proof}

\begin{proof}[Proof of conjecture \ref{legendre_conjecture} (Legendre)] 
Let $p_{n+1} - p_n < 2 \sqrt{p_{n+1}}$ be true, then according to lemma \ref{lemma_b_1} the interval $(m^2 - 2m,m^2)$ where $n = m^2$ contains a prime. Since $(m^2 - 2m,(m-1)^2)$ contains no integers and $(m-1)^2$ is not prime, then $((m-1)^2,m^2)$ contains a prime for every $m \geq 2$. 
\end{proof}

\section{Oppermann's conjecture}

\begin{conjecture}[Oppermann]\label{oppermann}
The interval $(n^2,(n+1)^2)$ contains two primes for any $n \in \mathbb{N}$.
\end{conjecture}




\begin{proposition}\label{prop_c_1}
The intervals $(l - \sqrt{l},l)$ and $(l,l+\sqrt{l})$ contain primes for every $l \geq p_{32} = 131$ if and only if $p_k - p_{k-1} < \sqrt{p_{k-1}}$ is true every prime $p_k \geq 131$.
\end{proposition}

\begin{proof}
Let $(l - \sqrt{l},l)$ and $(l,l + \sqrt{l})$ contain primes for every prime $l \geq p_{32} = 131$. Let $p$ and $q$ respectively belong to the intervals $(p_n - \sqrt{p_n},p_n)$, $(p_n,p_n  + \sqrt{p_n})$ where $l = p_n$. Since $p \leq p_{n-1} < p_n$ and $p_n < p_{n+1} \leq q$, so $p_{n-1}$ and $p_{n+1}$ also belong to $(p_n - \sqrt{p_n},p_n)$,$(p_n,p_n + \sqrt{p_n})$. Thus:
\begin{equation}
p_n - p_{n-1} < p_n - (p_n - \sqrt{p_n}) = \sqrt{p_n} , p_{n+1} - p_n < (p_n + \sqrt{p_n}) - p_n = \sqrt{p_n}
\end{equation}
and since $p_{32} - p_{31} < \sqrt{p_{31}}$, therefore $p_k - p_{k-1} < \sqrt{p_{k-1}}$ is true for every prime $p_k \geq p_{32}$.
\newline
\newline
Let $p_k - p_{k-1} < \sqrt{p_{k-1}}$ is true for every prime $p_k \geq p_{32}$ and $l=p_n \geq p_{32}$, then $p_n - p_{n-1} < \sqrt{p_{n-1}} < \sqrt{p_n}$ and also $p_{n+1} - p_n < \sqrt{p_n}$ hence $p_{n-1} \in (p_n - \sqrt{p_n},p_n)$ and $p_{n+1} \in (p_n,p_n + \sqrt{p_n})$. Thus we have that the intervals $(p_n -\sqrt{p_n},p_n)$,$(p_n,p_n + \sqrt{p_n})$ contain primes.
\end{proof}

\begin{proposition}\label{prop_c_2}
The intervals $(n - \sqrt{n},n)$ and $(n,n+\sqrt{n})$ contain primes for every integer $n \geq 131$ if and only if $p_k - p_{k-1} < \sqrt{p_{k-1}}$ is true for all primes $p_k \geq 131$.
\end{proposition}

\begin{proof}
Let $(n - \sqrt{n},n)$ and $(n,n+\sqrt{n})$ contain primes for all integers $n \geq 131$ and $n = p_k$. Then according proposition \ref{prop_c_1} $p_k - p_{k-1} < \sqrt{p_{k-1}}$ is true for all $p_k \geq 131$. Let $p_k - p_{k-1} < \sqrt{p_{k-1}}$ be true for all $p_k \geq 131$ but proposition \ref{prop_c_1} is false for some integer $n > p_{32} = 131$. Let $n_0$ be such an integer that at least one of the intervals $(n_0 - \sqrt{n_0},n_0)$, $(n_0,n_0 + \sqrt{n_0})$ contains no primes; then there are two cases:
\newline
\newline
Case 1: Let $(n_0 - \sqrt{n_0}, n_0)$ contain no primes. Let $p_{n-1}, p_n$ be such that $p_{32} \leq p_{n-1} < n_0 < p_n$, then the interval $(p_n - \sqrt{p_n}, p_n)$ contains no primes. Indeed, $n_0$ is not prime and the interval $(n_0 - \sqrt{n_0}, p_n) = (n_0 - \sqrt{n_0}, n_0)\cup[n_0]\cup[(n_0, p_n)$ contains no primes. Further we have $(p_n- \sqrt{p_n}, p_n)\subset(n_0 - \sqrt{n_0}, p_n)$ since $n_0 - \sqrt{n_0} < p_n - \sqrt{p_n}$ the interval $(p_n - \sqrt{p_n}, p_n)$ contains no primes.
\newline
\newline
Case 2: Let $(n_0, n_0 + \sqrt{n_0})$ contain no primes. Let $p_{n-1}, p_n$ be such that $p_{32}\leq p_{n-1} < n_0 < p_n$ then the interval $(p_{n-1}, p_{n-1} + \sqrt{p_{n-1}})$ contains no primes. Indeed, $n_0$ is not prime so the interval $(p_{n-1}, n_0 + \sqrt{n_0}) = (p_{n-1}, n_0)\cup[n_0]\cup(n_0, n_0 + \sqrt{n_0})$ contains no primes. Furthermore, $(p_{n-1}, p_{n-1}+ \sqrt{p_{n-1}})\subset(p_{n-1}, n_0 + \sqrt{n_0})$ since $p_{n-1}+ \sqrt{p_{n-1} }<n_0 + \sqrt{n_0}$ so the interval $(p_{n-1}, p_{n-1} + \sqrt{p_{n-1}})$ contains no primes. 
\newline
\newline
Both cases contradict proposition \ref{prop_c_1} since $p_{32} < p_n$ in case 1 and $p_{32} \leq p_{n-1}$ in case 2.
\end{proof}

\begin{proof}[Proof of conjecture \ref{oppermann} (Oppermann)]
Let $p_k - p_{k-1} < \sqrt{p_{k-1}}$ be true for every $p_k \geq 131$ then according to proposition \ref{prop_c_2} the intervals $(m^2 + m, (m+1)^2)$ where $n = (m+1)^2$ and $(m^2, m^2 + m)$ where $n = m^2, m^2 >131$ contain primes. The interval $(m^2, (m+1)^2)$ is a union of $(m^2, m^2 + m), (m^2 + m, (m+1)^2)$. Thus the conjecture is true for all $m^2$ not less than 131; by actual verification we find that it is true for smaller values. 
\end{proof}

\section{Discussion \& Conclusion}
The paper has explicitly shown that the general expected evaluation of the difference between consecutive primes $p_{n+1} - p_n = O(\sqrt{p_n})$ is a sufficient condition to prove Legendre's and Oppermann's conjectures. We have proved Legendre's and Oppermann's conjectures applying as evaluations of the difference between primes $p_{n+1} - p_n < 2\sqrt{p_{n+1}}$ and $p_{n+1} - p_n < \sqrt{p_n}$, respectively.

\bibliography{references}
\bibliographystyle{plain}

\end{document}